\documentclass[letterpaper, 10 pt, conference]{ieeeconf}  

\IEEEoverridecommandlockouts                              

\usepackage{cite}
\usepackage{amsmath}
\usepackage{amssymb}
\usepackage{balance}
\usepackage{graphicx}
\usepackage{subfigure}
\usepackage{epstopdf}
\usepackage{amsfonts}
\usepackage[all]{xy}
\usepackage{array,xcolor}
\usepackage{color}
\usepackage{bbm}
\usepackage{dsfont}
\usepackage{caption}
\usepackage{booktabs}
\usepackage{multirow}

\usepackage[font=footnotesize,labelfont=bf]{caption}

\usepackage{tikz}
\usetikzlibrary{shapes,matrix,decorations.shapes}
\usetikzlibrary{arrows,decorations.pathmorphing,backgrounds,positioning,fit,matrix}
\usetikzlibrary{shapes,arrows,chains}
\usetikzlibrary[calc]
\usetikzlibrary[plotmarks]
\usetikzlibrary[patterns,decorations]

\newcolumntype{C}[1]{>{\centering}m{#1}}
\newtheorem{definition}{Definition}
\newtheorem{theorem}{Theorem}

\newtheorem{lemma}{Lemma}

\newcommand{\tr}{{{\mathsf T}}}

\title{\LARGE \bf
Decomposition and Completion of Sum-of-Squares Matrices
}

\author{Yang Zheng$^{\dagger}$, Giovanni Fantuzzi$^{\ddagger}$, and Antonis~Papachristodoulou$^{\dagger}$
\thanks{The work of Y. Zheng is supported in part by the Clarendon Scholarship, and in part by the Jason Hu Scholarship. The work of A. Papachristodoulou is supported by the EPSRC project EP/M002454/1.}
\thanks{$^{\dagger} $Y. Zheng, and A. Papachristodoulou are with Department of Engineering Science at the University of Oxford. (E-mail:        \{yang.zheng, antonis\}@eng.ox.ac.uk)}%
\thanks{$^{\ddagger}$G. Fantuzzi is with Department of Aeronautics, Imperial College London, South Kensington Campus. (E-mail:        gf910@ic.ac.uk)}
}

\begin{document}

\maketitle
\thispagestyle{empty}
\pagestyle{empty}

\begin{abstract}
    This paper introduces a notion of decomposition and completion of \emph{sum-of-squares} (SOS) matrices.
    We show that a subset of sparse SOS matrices with chordal sparsity patterns can be equivalently decomposed into a sum of multiple SOS matrices that are nonzero only on a principal submatrix. 
    Also, the completion of an SOS matrix is equivalent to a set of SOS conditions on its principal submatrices and a consistency condition on the Gram representation of the principal submatrices.
    These results are partial extensions of chordal decomposition and completion of scalar matrices to matrices with polynomial entries.  We apply the SOS decomposition result to exploit sparsity in matrix-valued SOS programs. Numerical results demonstrate the high potential of this approach for solving large-scale sparse matrix-valued SOS programs.
\end{abstract}

\section{Introduction}

Matrix decomposition and completion naturally appear in a wide range of applications and have attracted considerable research attention~\cite{vandenberghe2014chordal,candes2009exact,chandrasekaran2011rank,agler1988positive,griewank1984existence,
grone1984positive,fukuda2001exploiting,andersen2010implementation,ZFPGW2017chordal,ZFPGW2016}. Problems involving sparse and positive semidefinite (PSD) matrices are of particular interest~\cite{vandenberghe2014chordal}. Consider a toy example
\begin{equation} \label{E:ExPSDdecomposition}
    \underbrace{\begin{bmatrix} 2 & 1 & 0 \\ 1 & 1 & 1 \\ 0 & 1 & 2 \end{bmatrix}}_{X \succeq 0} =  \underbrace{\begin{bmatrix} 2 & 1 & 0 \\ 1 & 0.5 & 0 \\ 0 & 0 & 0 \end{bmatrix}}_{X_1 \succeq 0} + \underbrace{\begin{bmatrix} 0 & 0 & 0 \\ 0 & 0.5 & 1 \\ 0 & 1 & 2 \end{bmatrix}}_{X_1 \succeq 0},
\end{equation}
where we decompose a sparse PSD matrix $X$ into a sum of two PSD matrices $X_1$, $X_2$, and each of them only consists of one PSD principal submatrix. In fact, this kind of decomposition always exists for a class of PSD matrices with chordal sparsity patterns (a precise definition will be given in Section~\ref{Section:Preliminaries})~\cite{agler1988positive,griewank1984existence}.

A concept related to the matrix decomposition above is that of \emph{PSD matrix completion}. Consider a symmetric matrix with partially specified entries
\begin{equation} \label{E:ExPSDcompletion}
   Z = \begin{bmatrix}
        2 & 1 & ? \\
        1 & 0.5 & 1\\
        ? & 1 & 2
    \end{bmatrix},
\end{equation}
where the question mark $?$ denotes unspecified entries. The PSD matrix completion problem asks whether $Z$ can be completed into a PSD matrix by filling in the unspecified entries. Clearly, a necessary condition for the existence of such a completion is that the principal submatrices are PSD. It turns out that this condition is also sufficient for matrices with chordal sparsity patterns~\cite{grone1984positive}. For the matrix in~\eqref{E:ExPSDcompletion}, it is not difficult to check that the principal submatrices are PSD, and there exists a PSD completion as follows
\begin{equation*}
  \begin{bmatrix}
        2 & 1  \\
        1 & 0.5 \\
    \end{bmatrix} \succeq 0, \begin{bmatrix}
        0.5 & 1  \\
        1 & 2 \\
    \end{bmatrix} \succeq 0,  Z = \begin{bmatrix}
        2 & 1 & 2 \\
        1 & 0.5 & 1\\
        2 & 1 & 2
    \end{bmatrix} \succeq 0.
\end{equation*}

The decomposition and completion results for sparse PSD matrices~\cite{agler1988positive,griewank1984existence,grone1984positive} actually allow us to equivalently replace a large PSD constraint by a set of coupled smaller PSD constraints, which promises better scalability for computations. This feature underpins much of the recent research on exploiting sparsity in conic programs, either by interior-point methods~\cite{fukuda2001exploiting,andersen2010implementation} or by first-order methods~\cite{ZFPGW2017chordal,ZFPGW2016,sun2014decomposition}. Also, applications of the decomposition and completion results have recently emerged in systems analysis and synthesis~\cite{andersen2014robust,ZMP2018Scalable}, as well as optimal power flow problems~\cite{dall2013distributed}.

In this paper, we provide a partial extension of the results in~\cite{agler1988positive,griewank1984existence,grone1984positive} to sparse matrices with polynomial entries. In other words, we consider the problem of decomposing and completing sparse polynomial matrices, where each entry is a polynomial with real coefficients in $n$ variables $x_1, \ldots, x_n$. For example, given a PSD polynomial matrix with the same pattern as that in~\eqref{E:ExPSDdecomposition}
\begin{equation} \label{E:ExSOSdecomposition}
    \begin{bmatrix}
        x^2+1 & x & 0\\
        x& x^2-2x+3 & x+1\\
        0 & x+1 & x^2+2
    \end{bmatrix} \succeq 0, \forall x \in \mathbb{R},
\end{equation}
we aim to answer whether it can be decomposed into a sum of two PSD polynomial matrices of the form
$$
    \underbrace{\begin{bmatrix}
        * & * & 0\\
        * & * & 0\\
        0 & 0 & 0
    \end{bmatrix}}_{\succeq 0} + \underbrace{\begin{bmatrix}
        0 & 0 & 0\\
        0 & * & *\\
        0 & * & *
    \end{bmatrix}}_{\succeq 0},
$$
where $*$ denotes a polynomial in $x$. Also, we try to address whether the following matrix can be completed into a PSD polynomial matrix
\begin{equation} \label{E:ExSOScompletion}
    \begin{bmatrix}
        x^2 + 1 & x & ? \\
        x& x^2 - 2x + 2.5& x + 1 \\
        ? & x+1& x^2 + 2
    \end{bmatrix},
\end{equation}
by replacing $?$ with an appropriate polynomial.

Note that checking the positive semidefiniteness of a given symmetric polynomial matrix is NP-hard in general~\cite{parrilo2003semidefinite}. This paper focuses on a subset of PSD matrices given by \emph{sum-of-squares (SOS)} matrices, as these can be identified with polynomial-time algorithms using semidefinite programs (SDPs)~\cite{gatermann2004symmetry,scherer2006matrix,kojima2003sums}.
%
Our motivation is the fact that SOS techniques represent a powerful tool for systems analysis, control, and optimzation (see, \emph{e.g.},~\cite{parrilo2003semidefinite, blekherman2012semidefinite}), but they do not scale well with problem size.
Most existing approaches to mitigate the scalability issue are based on the SOS representations of scalar polynomials; see an overview of recent advances in~\cite{ahmadi2017improving}. This paper describes sufficient conditions to decompose sparse SOS matrices into smaller ones, and for the existence of a SOS completion of a partial polynomial matrix. Thus, sparsity can be exploited to reduce the cost of computing with large and sparse SOS matrices.


The notion of decomposition and completion of SOS matrices has not been reported in the literature before. In this paper, we use hypergraphs to combine the Gram representation of SOS matrices~\cite{gatermann2004symmetry,scherer2006matrix,kojima2003sums} with the normal decomposition and completion results~\cite{agler1988positive,griewank1984existence, grone1984positive}. We prove that 1) the decomposition results for scalar matrices~\cite{agler1988positive,griewank1984existence} can be extended to a subset of sparse SOS matrices, and 2) the conditions for the existence of an SOS completion are similar to those for scalar matrices~\cite{grone1984positive}, with the addition of a consistency condition.
%
Due to the non-uniqueness of their  Gram matrix representation, however, our results of decomposition and completion for SOS matrices are not identical to those for scalar matrices in~\cite{agler1988positive,griewank1984existence, grone1984positive}.  As a direct application, we use the new decomposition result to exploit sparsity in matrix-value SOS programs. 
 Preliminary numerical results show the effectiveness of this approach.

The rest of this paper is organized as follows. Section~\ref{Section:Preliminaries} presents some preliminaries on chordal graphs. The results on decomposition and completion of sparse SOS matrices are given in Section~\ref{Section:SOSmatrices}. Section~\ref{Section:Application} discusses an application to matrix-valued SOS programs, including preliminary numerical examples. We conclude this paper in Section~\ref{Section:Conclusion}.

\section{Preliminaries on chordal graphs} \label{Section:Preliminaries}

The sparsity patterns for the matrices in~\eqref{E:ExPSDdecomposition} and~\eqref{E:ExPSDcompletion} can be commonly characterized by \emph{chordal graphs}. This section reviews chordal graphs and their relation to the decomposition and completion of sparse PSD matrices.

\subsection{Chordal graphs}
We define a graph $\mathcal{G}(\mathcal{V},\mathcal{E})$ by a set of nodes $\mathcal{V}=\{1,2,\dots, r\}$ and a set of edges $\mathcal{E} \subseteq \mathcal{V} \times \mathcal{V}$. Graph $\mathcal{G}(\mathcal{V},\mathcal{E})$ is called undirected if and only if $(i,j) \in \mathcal{E} \Leftrightarrow (j,i) \in \mathcal{E}$.  A \emph{cycle} of length $k$ is a sequence of nodes $\{v_1, v_2, \ldots, v_k\} \subseteq \mathcal{V}$ with $(v_k, v_{1}) \in \mathcal{E}$ and $(v_i, v_{i+1}) \in \mathcal{E}, \forall i = 1, \ldots, k-1$. A \emph{chord} in a cycle $\{v_1, v_2, \ldots, v_k\}$ is an edge $(v_i,v_j)$ that joins two nonconsecutive nodes in the cycle.

\begin{definition}[Chordal graph]
    An undirected graph is chordal if all its cycles of length at least four have a chord.
\end{definition}

Chordal graphs include some common classes of graphs, such as complete graphs, line graphs and trees. Fig.~\ref{F:ChordalGraph} illustrates some examples. We note that any non-chordal graph $\mathcal{G}(\mathcal{V},\mathcal{E})$ can always be extended to a chordal graph $\mathcal{G}(\mathcal{V},\hat{\mathcal{E}})$ by adding appropriate edges to $\mathcal{E}$~\cite{vandenberghe2014chordal}. Finally, we introduce the concept of \emph{cliques}: a \emph{clique} $\mathcal{C} \subseteq \mathcal{V}$ is a subset of nodes where $(i,j) \in \mathcal{E}, \forall i,j \in \mathcal{C}, i \neq j$. If a clique $\mathcal{C}$ is not a subset of  any other clique, then it is called a \emph{maximal clique}. For example, in Fig.~\ref{F:ChordalGraph}(a), there are two maximal cliques, $\mathcal{C}_1 = \{1, 2\}$ and $\mathcal{C}_2 = \{2, 3\}$.

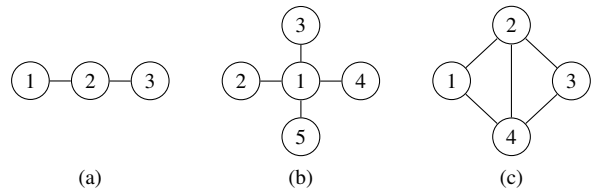
\begin{figure}[t]
    \centering
    \footnotesize
    \begin{tikzpicture}
	  \matrix (m) [matrix of nodes,
	  		       row sep = 0.8em,	
	  		       column sep = 1em,	
  			       nodes={circle, draw=black}] at (-2.8,0)
  		{ &  & \\ 1 & 2 & 3 \\ & &\\};
		\draw (m-2-1) -- (m-2-2);
		\draw (m-2-2) -- (m-2-3);
		\node at (-2.8,-1.3) {(a)};
		\matrix (m2) [matrix of nodes,
	  		       row sep = 0.8em,	
	  		       column sep = 1em,	
  			       nodes={circle, draw=black}] at (0,0)
        { & 3 & \\ 2 & 1 & 4 \\& 5 &\\};
		\draw (m2-1-2) -- (m2-2-2);
		\draw (m2-2-1) -- (m2-2-2);
		\draw (m2-2-2) -- (m2-2-3);
		\draw (m2-2-2) -- (m2-3-2);
		\node at (0,-1.3) {(b)};

		\matrix (m3) [matrix of nodes,
	  		       row sep = 0.8em,	
	  		       column sep = 1.em,	
  			       nodes={circle, draw=black}] at (2.8,0)
        { & 2 & \\ 1 &  & 3 \\& 4 &\\};
		\draw (m3-1-2) -- (m3-2-1);
		\draw (m3-1-2) -- (m3-2-3);
		\draw (m3-2-1) -- (m3-3-2);
		\draw (m3-2-3) -- (m3-3-2);
        \draw (m3-1-2) -- (m3-3-2);
		\node at (2.8,-1.3) {(c)};
	\end{tikzpicture}
    \caption{Examples of chordal graphs: (a) a line graph; (b) a star graph; (c) a triangulated graph.}
    \label{F:ChordalGraph}
\end{figure}

\subsection{Chordal decomposition and completion}

Given an undirected graph $\mathcal{G}(\mathcal{V},\mathcal{E})$, we define an extended set of edges $\mathcal{E}^* = \mathcal{E} \cup\{(i,i), \forall i \in \mathcal{V}\}$ that includes all self-loops. Then, we define the space of symmetric matrices as
\begin{equation} \label{E:SparseSymMatrix}
    \mathbb{S}^r(\mathcal{E},0) := \{X \in \mathbb{S}^{r} | X_{ij} = X_{ji}= 0\; \text{if} \; (j,i) \notin {\mathcal{E}}^* \},
\end{equation}
and the cone of sparse PSD matrices as
\begin{equation}
    \mathbb{S}^r_+(\mathcal{E},0) := \{X \in \mathbb{S}^{r}(\mathcal{E},0) | X \succeq 0 \}.
\end{equation}
Also, we denote by $\mathbb{S}^r_+(\mathcal{E},?)$ the set of matrices in $\mathbb{S}^r(\mathcal{E},0)$ that have a PSD completion, \emph{i.e.},
\begin{equation}\label{E:ConePSD}
    \mathbb{S}^r_+(\mathcal{E},?) := \mathbb{P}_{\mathbb{S}^r(\mathcal{E},0)}(\mathbb{S}^r_{+}),
\end{equation}
where $\mathbb{P}_{\mathbb{S}^r(\mathcal{E},0)}$ denotes the projection onto $\mathbb{S}^r(\mathcal{E},0)$.
Given a maximal clique $\mathcal{C}_k$, we define a matrix $E_{\mathcal{C}_k} \in \mathbb{R}^{|\mathcal{C}_k| \times r}$ as
\begin{equation} \label{E:IndexMatrix}
    (E_{\mathcal{C}_k})_{ij} := \begin{cases} 1, \quad \text{if } \mathcal{C}_k(i) = j, \\ 0, \quad \text{otherwise}, \end{cases}
\end{equation}
where $|\mathcal{C}_k|$ denotes the number of nodes in $\mathcal{C}_k$, and $\mathcal{C}_k(i)$ denotes the $i$-th node in $\mathcal{C}_k$, sorted in the natural ordering. Note that $X_k = E_{\mathcal{C}_k}XE_{\mathcal{C}_k}^\tr \in \mathbb{S}^{|\mathcal{C}_k|}$ extracts a principal submatrix defined by the indicies in clique $\mathcal{C}_k$, and the operation $E_{\mathcal{C}_k}^\tr X_kE_{\mathcal{C}_k}$ {inflates} a $\vert\mathcal{C}_k\vert \times \vert \mathcal{C}_k\vert $ matrix into a sparse $r\times r$ matrix.
Then, the following results characterize, respectively, the membership to the sets $\mathbb{S}^n_{+}(\mathcal{E},0)$ and $\mathbb{S}^n_+(\mathcal{E},?)$ when the underlying graph $\mathcal{G}(\mathcal{V},\mathcal{E})$ is chordal.

\begin{theorem}[\!\!\cite{agler1988positive,griewank1984existence}]\label{T:ChordalDecompositionTheorem}
     Let $\mathcal{G}(\mathcal{V},\mathcal{E})$ be a chordal graph with maximal cliques $\{\mathcal{C}_1,\mathcal{C}_2, \ldots, \mathcal{C}_t\}$. Then, $X\in\mathbb{S}^r_{+}(\mathcal{E},0)$ if and only if there exist $X_k \in \mathbb{S}^{\vert \mathcal{C}_k \vert}_+, k=1,\,\ldots,\,t$, such that
\end{theorem}
 \begin{equation*} 
        X = \sum_{k=1}^{t} E_{\mathcal{C}_k}^\tr X_k E_{\mathcal{C}_k}.
 \end{equation*}
\begin{theorem} [{\!\!\cite{grone1984positive}}]\label{T:ChordalCompletionTheorem}
     Let $\mathcal{G}(\mathcal{V},\mathcal{E})$ be a chordal graph with maximal cliques $\{\mathcal{C}_1,\mathcal{C}_2, \ldots, \mathcal{C}_t\}$. Then, $Z\in\mathbb{S}^r_+(\mathcal{E},?)$ if and only if
\end{theorem}
    $$
        E_{\mathcal{C}_k} Z E_{\mathcal{C}_k}^\tr \in \mathbb{S}^{\vert \mathcal{C}_k \vert}_+, \quad k=1,\,\ldots,\,t.
    $$

Given a chordal graph $\mathcal{G}(\mathcal{V},\mathcal{E})$, according to Theorem~\ref{T:ChordalDecompositionTheorem}, a sparse PSD matrix $X \in \mathbb{S}^r_{+}(\mathcal{E},0)$ can always be written as a sum of multiple PSD matrices that are nonzero only on a principal submatrix. For example, the matrix $X$ in~\eqref{E:ExPSDdecomposition} has a sparsity pattern corresponding to Fig.~\ref{F:ChordalGraph}(a). Then, we have
$$
    E_{\mathcal{C}_1} = \begin{bmatrix} 1&0&0 \\0& 1& 0  \end{bmatrix}, E_{\mathcal{C}_2} = \begin{bmatrix} 0&1&0 \\0& 0& 1  \end{bmatrix},
$$
and
$$
     \begin{bmatrix} 2 & 1 & 0 \\ 1 & 1 & 1 \\ 0 & 1 & 2 \end{bmatrix} =  E_{\mathcal{C}_1}^\tr\underbrace{\begin{bmatrix} 2 & 1  \\ 1 & 0.5  \end{bmatrix}}_{\succeq 0}E_{\mathcal{C}_1}  + E_{\mathcal{C}_2}^\tr\underbrace{\begin{bmatrix} 0.5 & 1 \\  1 & 2 \end{bmatrix}}_{\succeq 0}E_{\mathcal{C}_2}.
$$
Similarly, Theorem 2 states that the matrix $X$ in (2) has a PSD completion if and only if the $2 \times 2$ principal submatrices $E_{\mathcal{C}_1} Z E_{\mathcal{C}_1}^\tr$ and $E_{\mathcal{C}_2} Z E_{\mathcal{C}_2}^\tr$ are PSD, which is easy to verify. These simple examples illustrate that constraints of the form $X \in \mathbb{S}_+(\mathcal{E},0)$ or $Z \in \mathbb{S}_+(\mathcal{E},?)$ can be replaced by PSD constraints on certain principal submatrices, provided that its sparsity patten is chordal. 
%
This feature has been exploited successfully to improve the scalability of solving large-scale sparse SDPs in~\cite{fukuda2001exploiting, andersen2010implementation, ZFPGW2017chordal, ZFPGW2016,sun2014decomposition}.


%

\section{Decomposition \& completion of SOS matrices} \label{Section:SOSmatrices}

In this section, we apply Theorems~\ref{T:ChordalDecompositionTheorem} and~\ref{T:ChordalCompletionTheorem} to the decomposition and completion of a class of PSD matrices with polynomial entries.

\subsection{Nonnegativity and sum-of-squares}

We denote by $\mathbb{R}[x]_{n,2d}$ the set of polynomials in $n$ variables with real coefficients of degree no more than $2d$.  The set of $q \times r$ polynomial matrices with entries in $\mathbb{R}[x]_{n,2d}$ is denoted by $\mathbb{R}[x]_{n,2d}^{q \times r}$. A polynomial $p(x) \in \mathbb{R}[x]_{n,2d}$ is \emph{nonnegative} or PSD if $p(x) \geq 0, \forall\, x \in \mathbb{R}^n$, and a symmetric polynomial matrix $P(x) \in \mathbb{R}[x]_{n,2d}^{r \times r}$ is PSD if $P(x) \succeq 0, \forall\, x \in \mathbb{R}^n$.

Note that checking positive semidefiniteness of a polynomial $p(x)$ or a polynomial matrix $P(x)$ is NP-hard in general~\cite{parrilo2003semidefinite}.
%
A popular tractable approach is to replace the PSD constraint by a sum-of-squares (SOS) constraint. We say that $p(x)\in \mathbb{R}[x]_{n,2d}$ is an SOS polynomial if there exists polynomials $f_i(x) \in \mathbb{R}[x]_{n,d}, i = 1, \ldots, s$ such that
$
	p(x) = \sum_{i=1}^s f^2_i(x).
$
Also, we define an SOS matrix as follows~\cite{gatermann2004symmetry,scherer2006matrix,kojima2003sums}.
\begin{definition}
    A symmetric polynomial matrix $P(x) \in \mathbb{R}[x]_{n,2d}^{r \times r}$ is an SOS matrix if there exists a polynomial matrix $M(x) \in \mathbb{R}[x]_{n,d}^{s \times r}$ such that $P(x) = M^\tr(x)M(x)$.
\end{definition}

Clearly, the existence of an SOS representation ensures the positive semidefiniteness of $p(x)$ or $P(x)$. For simplicity, we denote by $\Sigma_{n,2d}^{r}$ the set of $r \times r$ SOS matrices with entries in $\mathbb{R}[x]_{n,2d}$. It is known that the problem of checking membership of $\Sigma_{n,2d}^{r}$ can be cast as an SDP.
\begin{lemma}[\!\!\cite{gatermann2004symmetry,scherer2006matrix,kojima2003sums}] \label{Lemma:SOSlemma}
    $P(x) \in \Sigma_{n,2d}^{r}$ if and only if there exists a PSD matrix $Q \in \mathbb{S}^l_+$ with $l= r\times N $ and $N = {n+d \choose d}$ such that
    \begin{equation} \label{E:matrixSOS}
        P(x) = \left(I_r \otimes v_d(x)\right)^\tr  Q \left(I_r \otimes v_d(x)\right),
    \end{equation}
    where $v_d(x) = [ 1,x_1,x_2,\ldots,x_n,x_1^2,x_1x_2,\ldots,x_n^d ]^\tr$ is the standard monomial vector of degree up to $d$.
\end{lemma}


The matrix $Q$ in~\eqref{E:matrixSOS} is called the \emph{Gram matrix} of the SOS representation, which is usually not unique.

\subsection{Decomposition of sparse SOS matrices}

Similar to the sparse scalar matrix case~\eqref{E:SparseSymMatrix}, we define sparse SOS matrices characterized by an undirected graph $\mathcal{G}(\mathcal{V},\mathcal{E})$,
\begin{equation*}
    \begin{aligned}
        \Sigma_{n,2d}^r(\mathcal{E},0) = \bigg\{P(x) \in \Sigma_{n,2d}^r  \bigg\vert  & p_{ij}(x) = p_{ji}(x)=0, \\
          &\qquad \quad \; \text{if } (i,j) \notin \mathcal{E}^* \bigg\}.
    \end{aligned}
\end{equation*}

Given a sparse SOS matrix $P(x)$, according to Lemma~\ref{Lemma:SOSlemma}, its SOS representation can be written as
$$
    P(x) = \begin{bmatrix} v_d(x)^\tr Q_{11} v_d(x) & \ldots & v_d(x)^\tr Q_{1r} v_d(x) \\
    \vdots & \ddots & \vdots \\
    v_d(x)^\tr Q_{r1} v_d(x) & \ldots & v_d(x)^\tr Q_{rr} v_d(x)  \end{bmatrix},
$$
where $Q_{ij} \in \mathbb{R}^{N \times N}$, $i,j = 1,\ldots,r$ is the $(i,j)$-th block of matrix $Q \in \mathbb{S}^{l}_+$. If $ P(x)\in \Sigma_{n,2d}^r(\mathcal{E},0)$, then we have
\begin{equation} \label{E:ZeroPolynomial}
    p_{ij}(x) = v_d(x)^\tr Q_{ij} v_d(x) = 0, \text{if } (i,j) \notin \mathcal{E}^*.
\end{equation}
Still, $Q_{ij}$ may be a nonzero matrix in~\eqref{E:ZeroPolynomial}. Note that while $Q$ is a symmetric matrix, the $ij$-th off-diagonal block $Q_{ij}$ need not be so. This means the Gram matrix $Q$ for a sparse SOS matrix $ P(x)\in \Sigma_{n,2d}^r(\mathcal{E},0)$ can be dense. To maintain the sparsity of $P(x)$ in the Gram matrix $Q$, we consider a subset of SOS matrices
\begin{equation*}
    \begin{aligned}
        \widetilde{\Sigma}_{n,2d}^r&(\mathcal{E},0) = \bigg\{P(x) \in \Sigma_{n,2d}^r(\mathcal{E},0) \bigg\vert  P(x) \text{~admits a} \\
        & \text{Gram matrix~} Q \succeq 0, \text{~with~} Q_{ij} = 0 \text{~when~} p_{ij}(x) = 0 \bigg\}.
    \end{aligned}
\end{equation*}

With this restriction, we have the following result.
\begin{theorem}[SOS matrix decomposition]\label{T:SOSDecompositionTheorem}
     Let $\mathcal{G}(\mathcal{V},\mathcal{E})$ be a chordal graph with maximal cliques $\{\mathcal{C}_1,\mathcal{C}_2, \ldots, \mathcal{C}_t\}$. Then, $P(x) \in \widetilde{\Sigma}_{n,2d}^r(\mathcal{E},0)$ if and only if there exist $P_k(x) \in \Sigma_{n,2d}^{|\mathcal{C}_k|}, k=1,\,\ldots,\,t$, such that
\end{theorem}
 \begin{equation*} 
        P(x)= \sum_{k=1}^{t} E_{\mathcal{C}_k}^\tr P_k(x) E_{\mathcal{C}_k}.
 \end{equation*}

\begin{proof}
    The proof, organized in three steps, combines the Gram representation of SOS matrices with Theorem~\ref{T:ChordalDecompositionTheorem}.

    \emph{Step 1 (Sparse Gram matrix):} We observe that $P(x) \in \widetilde{\Sigma}_{n,2d}^r(\mathcal{E},0)$ if and only if it admits a sparse Gram matrix $Q \in \mathbb{S}^{rN}_+$, where the $(i,j)$-th block $Q_{ij}=Q_{ji}^\tr = 0, \forall \, (i,j) \notin \mathcal{E}^*$. This means the Gram matrix $Q \in \mathbb{S}^{rN}_+$ has a block sparsity pattern defined by $\mathcal{G}(\mathcal{V},\mathcal{E})$. In fact, as shown in~\cite{ZKSP2017scalable}, this kind of sparsity pattern is also chordal. 

    To be precise, we define a hyper-graph $\widetilde{\mathcal{G}}(\widetilde{\mathcal{V}},\widetilde{\mathcal{E}})$ with a hyper-node set defined as
    $$
        \widetilde{\mathcal{V}} = \{\underbrace{1,\ldots,N}_{1}, \underbrace{N+1, \ldots, 2N}_2, \ldots, \underbrace{(r-1)N+1, \ldots ,rN}_{r}\},
    $$
    and a hyper-edge set defined as
    $
        \widetilde{\mathcal{E}} = \bigcup_{k=1}^t \widetilde{\mathcal{C}}_k \times \widetilde{\mathcal{C}}_k,
    $
    where the hyper-clique $\widetilde{\mathcal{C}}_k$ is defined as
    \begin{equation} \label{E:HyperClique}
        \widetilde{\mathcal{C}}_k = \{(j-1)N+1,\ldots,jN \mid j \in \mathcal{C}_k\}, k = 1, \ldots, t.
    \end{equation}
    Then, the sparsity pattern of the Gram matrix $Q$ can be described by $\widetilde{\mathcal{G}}(\widetilde{\mathcal{V}},\widetilde{\mathcal{E}})$, \emph{i.e.}, $Q \in \mathbb{S}_+^{rN}(\widetilde{\mathcal{E}},0)$.

    Moreover, if $\mathcal{G}(\mathcal{V},\mathcal{E})$ is chordal with its maximal cliques as $\mathcal{C}_1, \ldots, \mathcal{C}_p$, it is shown in~\cite{ZKSP2017scalable} that the hyper-graph $\widetilde{\mathcal{G}}(\widetilde{\mathcal{V}},\widetilde{\mathcal{E}})$ is also chordal with a set of maximal cliques $\widetilde{\mathcal{C}}_1, \ldots, \widetilde{\mathcal{C}}_t$.

    \emph{Step 2 (Block chordal decomposition):} According to Theorem~\ref{T:ChordalDecompositionTheorem}, the Gram matrix $Q \in \mathbb{S}_+^{rN}(\widetilde{\mathcal{E}},0)$ if and only if there exists a decomposition
    $$
        Q = \sum_{k=1}^t {E}_{\widetilde{\mathcal{C}}_k}^TQ_k{E}_{\widetilde{\mathcal{C}}_k}, \vspace{-1mm}
    $$
    where $Q_k \in \mathbb{S}_+^{|\widetilde{\mathcal{C}}_k|}, k = 1, \ldots, t$. Also, combining~\eqref{E:IndexMatrix} with~\eqref{E:HyperClique}, it is not difficult to see
    \begin{equation} \label{E:IndexMatrixWide}
        {E}_{\widetilde{\mathcal{C}}_k} =  {E}_{{\mathcal{C}}_k} \otimes I_N, k =1, \ldots, t.
    \end{equation}

    \emph{Step 3 (SOS matrix decomposition):} 
    In the context of SOS matrices, we have
    \begin{equation} \label{E:SOSdecomposition_s1}
        \begin{aligned}
            P(x) &= \left(I_r \otimes v_d(x)\right)^\tr  Q \left(I_r \otimes v_d(x)\right) \\
                & = \left(I_r \otimes v_d(x)\right)^\tr  \left(\sum_{k=1}^t {E}_{\widetilde{\mathcal{C}}_k}^TQ_k{E}_{\widetilde{\mathcal{C}}_k}\right) \left(I_r \otimes v_d(x)\right) \\
                & = \sum_{k=1}^t \left[ \left(I_r \otimes v_d(x)\right)^\tr {E}_{\widetilde{\mathcal{C}}_k}^TQ_k{E}_{\widetilde{\mathcal{C}}_k} \left(I_r \otimes v_d(x)\right) \right],
        \end{aligned}
    \end{equation}
    Furthermore, using the properties of the Kronecker product and~\eqref{E:IndexMatrixWide} we obtain
    \begin{equation} \label{E:IndexMatrixSOS}
        \begin{aligned}
        {E}_{\widetilde{\mathcal{C}}_k} \left(I_r \otimes v_d(x)\right) &= \left( {E}_{{\mathcal{C}}_k} \otimes I_N \right) \cdot \left(I_r \otimes v_d(x)\right) \\
        & = {E}_{{\mathcal{C}}_k} \otimes v_d(x) \\
        & = \left( I_{|\mathcal{C}_k|} \otimes v_d(x) \right) \cdot \left({E}_{{\mathcal{C}}_k} \otimes 1\right) \\
        & = \left( I_{|\mathcal{C}_k|} \otimes v_d(x) \right) \cdot {E}_{{\mathcal{C}}_k}.
        \end{aligned}
    \end{equation}

    Substituting~\eqref{E:IndexMatrixSOS} into~\eqref{E:SOSdecomposition_s1} yields
    $$
        \begin{aligned}
                P(x) & = \sum_{k=1}^t \left[  {E}_{{\mathcal{C}}_k}^\tr  \left( I_{|\mathcal{C}_k|} \otimes v_d(x) \right)^\tr Q_k\left( I_{|\mathcal{C}_k|} \otimes v_d(x) \right) {E}_{{\mathcal{C}}_k} \right] \\
                & = \sum_{k=1}^t {E}_{{\mathcal{C}}_k}^\tr  P_k(x) {E}_{{\mathcal{C}}_k},
            \end{aligned}
    $$
    where $P_k(x) \in \Sigma_{n,2d}^{|\mathcal{C}_k|} , k = 1,\ldots, t$.

\end{proof}

The proof of Theorem~\ref{T:SOSDecompositionTheorem} is based on the perspective of hyper-graphs, combining the Gram representation of SOS matrices (\emph{i.e.}, Lemma~\ref{Lemma:SOSlemma}) with the normal chordal decomposition result (\emph{i.e.}, Theorem~\ref{T:ChordalDecompositionTheorem}). Note that Theorem~\ref{T:SOSDecompositionTheorem} offers a necessary and sufficient condition for checking the membership of $\widetilde{\Sigma}_{n,2d}^r(\mathcal{E},0)$, which is a subset of SOS matrices since $p_{ij}(x) = 0$ does not require $Q_{ij} = 0$ in general.

Given $P(x) \in \widetilde{\Sigma}_{n,2d}^r(\mathcal{E},0)$, we can construct the decomposed SOS matrices $P_k(x)$:
\begin{enumerate}
  \item Find a sparse Gram matrix $Q \in \mathbb{S}^{rN}_+(\widetilde{\mathcal{E}},0)$;
  \item Perform a normal chordal decomposition $Q = \sum_{k=1}^t {E}_{\widetilde{\mathcal{C}}_k}^TQ_k{E}_{\widetilde{\mathcal{C}}_k}$ (\emph{e.g.},~\cite[Chapter 9]{vandenberghe2014chordal});
  \item Then, the decomposed SOS matrices are $P_k(x) = \left(I_{|\mathcal{C}_k|} \otimes v_d(x) \right)^\tr Q_k\left( I_{|\mathcal{C}_k|} \otimes v_d(x) \right), k = 1, \ldots, t$.
\end{enumerate}

For the example shown in~\eqref{E:ExSOSdecomposition}, the monomial basis is $v_d(x) = [1,x]^\tr$, and we can find a sparse Gram matrix
$$
	Q = \left[
	\begin{array}{cc|cc|cc}
	1 & 0 & 0 & 0.4 & 0 & 0\\
	0 & 1 & 0.6 & 0 & 0 & 0\\
	\hline
	0 & 0.6 & 3 & -1 & 1 & 0.8\\
	0.4 & 0 & -1 & 1 & 0.2 & 0\\
    \hline	
    0 & 0 &1 & 0.2 & 2 &  0 \\
    0 & 0 &0.8 & 0 & 0  & 1
	\end{array}
	\right] \in \mathbb{S}^{6}_+(\widetilde{\mathcal{E}},0).
$$
For this particular matrix $Q$, we have a decomposition
$$
\begin{aligned}
    Q_1 &=
	\left[
	\begin{array}{cc|cc}
	1 & 0 & 0 & 0.4 \\
	0 & 1 & 0.6 & 0 \\
	\hline
	0 & 0.6 & 1.11 & -0.545 \\
	0.4 & 0 & -0.545 & 0.56 \\
	\end{array}
	\right] \in \mathbb{S}^{4}_+. \\
    Q_2 &= \left[
	\begin{array}{cc|cc}
	 1.89 & -0.455 & 1 & 0.8\\
	 -0.455 & 0.44 & 0.2 & 0\\
    \hline	
    1 & 0.2 & 2 &  0 \\
    0.8 & 0 & 0  & 1
	\end{array}
	\right] \in \mathbb{S}^{4}_+.
\end{aligned}
$$
Then, an SOS decomposition for~\eqref{E:ExSOSdecomposition} is given as
$$
    \begin{aligned}
        P_1(x) &= \begin{bmatrix}
            x^2+1 & x \\
            x& 0.56x^2-1.09x+1.11 \\
        \end{bmatrix} \in \Sigma_{1,2}^2 \\
        P_2(x) &=\begin{bmatrix}
        0.44x^2-0.91x+1.89 & x+1\\
        x+1 & x^2+2
    \end{bmatrix}  \in \Sigma_{1,2}^2
    \end{aligned}
$$

Here, we emphasize that the main interest of Theorem~\ref{T:SOSDecompositionTheorem} is not on computing an actual SOS decomposition. Instead, this theorem offers a computationally efficient way to check
if a matrix $P(x)$ belongs to $\widetilde{\Sigma}_{n,2d}^{r}(\mathcal{E},0)$,
which can enable the solution of large matrix-valued SOS programs (see Section~\ref{Section:Application}).

\subsection{Completion of sparse SOS matrices}

Here, we give an analogue result to Theorem~\ref{T:ChordalCompletionTheorem} for partial SOS matrices. Given a graph $\mathcal{G}(\mathcal{V},\mathcal{E})$, we say $P(x)$ is a partial symmetric polynomial matrix if $p_{ij}(x)=p_{ji}(x)$ are given when $(i,j) \in \mathcal{E}^*$.  Moreover, we say that $F(x)$ is an SOS completion of the partial symmetric matrix $P(x)$ if $F(x)$ is SOS and $F_{ij}(x) = P_{ij}(x)$ when $(i, j) \in \mathcal{E}^*$. Precisely, we define a set of SOS completable matrices as
\begin{equation*}
    \begin{aligned}
        \Sigma_{n,2d}^r(\mathcal{E},?)=\bigg\{P(x): \mathbb{R}^n &\to \mathbb{S}^r(\mathcal{E},0) \big\vert  \exists F(x) \in \Sigma_{n,2d}^r,  \\
         & \,F_{ij}(x) = P_{ij}(x),\; \forall (i,j) \in \mathcal{E}^*\bigg\}.
    \end{aligned}
\end{equation*}

For instance, the matrix in~\eqref{E:ExSOScompletion} is a partial symmetric polynomial matrix defined by the graph in Fig.~\ref{F:ChordalGraph}(a), and we will show below that matrix is also SOS completable.

\begin{theorem}[SOS matrix completion] \label{T:SOSCompletion}
        Let $\mathcal{G}(\mathcal{V},\mathcal{E})$ be a chordal graph with maximal cliques $\{\mathcal{C}_1,\mathcal{C}_2, \ldots, \mathcal{C}_t\}$. Then, $P(x) \in \Sigma_{n,2d}^r(\mathcal{E},?)$ if and only if
        \begin{equation} \label{E:SOScondition}
             P_k(x) = E_{\mathcal{C}_k}P(x)E_{\mathcal{C}_k}^\tr \in \Sigma_{n,2d}^{|\mathcal{C}_k|}, \; k = 1, \ldots,t,
        \end{equation}
        and the Gram matrix $Q_k$ corresponding to each $P_k(x)$ satisfies a consistency condition: elements of $Q_k$ that map to the same entries of the global Gram matrix $Q$, which represents the original polynomial matrix $P(x)$, are identical. Mathematically, it requires
        \begin{equation} \label{E:CommonGram}
            \begin{aligned}
                {E}_{\widetilde{\mathcal{C}}_i \cap \widetilde{\mathcal{C}}_j} & \left( {E}_{\widetilde{\mathcal{C}}_i}^\tr Q_i {E}_{\widetilde{\mathcal{C}}_i} - {E}_{\widetilde{\mathcal{C}}_j}^\tr Q_j {E}_{\widetilde{\mathcal{C}}_j}\right) {E}_{\widetilde{\mathcal{C}}_i \cap \widetilde{\mathcal{C}}_j}^\tr \\
                &\qquad \qquad \qquad \qquad \quad = 0, \forall \widetilde{\mathcal{C}}_i \cap \widetilde{\mathcal{C}}_j \neq \emptyset,
            \end{aligned}
        \end{equation}
        where $\widetilde{\mathcal{C}}_i, i = 1, \ldots, t$ are the hyper-cliques defined in~\eqref{E:HyperClique}.
\end{theorem}

\begin{proof}
    Similar to the proof of Theorem~\ref{T:SOSDecompositionTheorem}, we rely on the hyper-graph $\widetilde{\mathcal{G}}(\widetilde{\mathcal{V}},\widetilde{\mathcal{E}})$, which is chordal with a set of maximal cliques $\widetilde{\mathcal{C}}_1, \ldots, \widetilde{\mathcal{C}}_t$.

    \emph{$\Leftarrow$:} If $P(x)\in \Sigma_{n,2d}^r(\mathcal{E},?)$, then there exists a SOS matrix $F(x)$ with a Gram matrix $Q \in \mathbb{S}^{rN}_+$, such that
    $$
        P_k(x) = E_{\mathcal{C}_k}P(x)E_{\mathcal{C}_k}^\tr = E_{\mathcal{C}_k}F(x)E_{\mathcal{C}_k}^\tr, k = 1, \ldots t.
    $$
    Also, using a property similar to~\eqref{E:IndexMatrixSOS}, we have
    $$
        \begin{aligned}
            E_{\mathcal{C}_k}F(x)E_{\mathcal{C}_k}^\tr &= E_{\mathcal{C}_k}\left( (I_r \otimes v_d(x))^\tr Q (I_r \otimes v_d(x))\right)E_{\mathcal{C}_k}^\tr \\
            & =  (I_{|\mathcal{C}_k|} \otimes v_d(x))^\tr Q_k (I_{|\mathcal{C}_k|} \otimes v_d(x)),
        \end{aligned}
    $$
    where
    \begin{equation} \label{E:subGramMatrix}
        Q_k = E_{\widetilde{\mathcal{C}}_k}QE_{\widetilde{\mathcal{C}}_k}^\tr \in \mathbb{S}^{|\mathcal{C}_k|N}_+.
    \end{equation}
    Therefore, $P_k(x) \in \Sigma_{n,2d}^{|\mathcal{C}_k|}$, and the Gram matrices $Q_k$ in~\eqref{E:subGramMatrix} satisfy the \emph{consistency Gram matrix} condition~\eqref{E:CommonGram}.

    \emph{$\Rightarrow$:} If we have conditions~\eqref{E:SOScondition} and~\eqref{E:CommonGram}, then we can form a partial symmetric matrix $Q$ with $Q_k= E_{\widetilde{\mathcal{C}}_k}QE_{\widetilde{\mathcal{C}}_k}^\tr, k=1,\ldots,t$. Since $Q_k \succeq 0$ and graph $\widetilde{\mathcal{G}}(\widetilde{\mathcal{V}},\widetilde{\mathcal{E}})$ is chordal, Theorem~\ref{T:ChordalCompletionTheorem} ensures that
    $
        Q \in \mathbb{S}_+^{rN}(\widetilde{\mathcal{E}},?).
    $
    Then we can find a PSD completion $\hat{Q} \succeq 0$ for the partial symmetric matrix $Q$. In the context of polynomial matrices, we have found an SOS completion $F(x)$ for $P(x)$, \emph{i.e.},
    \begin{equation} \label{E:SOScompletion}
        F(x) = (I_r \otimes v_d(x))^\tr \hat{Q} (I_r \otimes v_d(x)) \in \Sigma_{n,2d}^r.
    \end{equation}
    Therefore, $P(x) \in \Sigma_{n,2d}^r(\mathcal{E},?)$.
\end{proof}

The proof of Theorem~\ref{T:SOSCompletion} is similar to that of Theorem~\ref{T:SOSDecompositionTheorem}, both of which utilize the perspective of hypergraphs and then apply the normal results of chordal decomposition and completion. Given $P(x) \in \Sigma_{n,2d}^r(\mathcal{E},?)$, we can find an SOS completion using the following steps:
\begin{enumerate}
  \item Find a PSD completable Gram matrix, $Q \in \mathbb{S}^{rN}_+(\widetilde{\mathcal{E}},?)$;
  \item Find a PSD completion $\hat{Q}$ using any PSD completion algorithm (\emph{e.g.},~\cite[Chapter 10]{vandenberghe2014chordal} and~\cite[Section 2]{fukuda2001exploiting});
  \item Then, an SOS matrix is given by~\eqref{E:SOScompletion}.
\end{enumerate}

Using this procedure, we are able to find an SOS completion for the example~\eqref{E:ExSOScompletion}:
$$
{
    \begin{bmatrix}
        x^2 + 1 & x & 0.3x^2+0.6x+0.3 \\
        x& x^2 - 2x + 2.5& x + 1 \\
        0.3x^2+0.6x+0.3 & x+1& x^2 + 2
    \end{bmatrix}\!\!.
    }
$$

\vspace{0.5ex}
\section{Application to matrix-valued SOS programs} \label{Section:Application}

The SOS decomposition can be readily applied to exploit sparsity in matrix-valued SOS programs. Consider the following matrix-valued SOS program
\begin{equation} \label{E:matrixSOSprogram}
    \begin{aligned}
        \min_{u}\quad & w^\tr u \\[-1ex]
        \text{subject to} \quad  & P(x) = P_0(x) - \sum_{i=1}^h u_iP_i(x),
        \\ & P(x) \text{ is SOS},
    \end{aligned}
\end{equation}
where $u \in \mathbb{R}^h$ is the decision variable, $w \in \mathbb{R}^h$ defines a linear objective function, and $P_0(x),\,\ldots,\,P_t(x)$ are given $r \times r$ symmetric polynomial matrices with a common sparsity pattern $\mathcal{E}$. Note that matrix-valued SOS programs have found applications in robust semidefinite programs~\cite{scherer2006matrix} and control theory~\cite{peet2009positive}.
We assume that $\mathcal{E}$ is chordal; otherwise, a chordal extension can be found~\cite{vandenberghe2014chordal}. Clearly,~\eqref{E:matrixSOSprogram} is equivalent to
\begin{equation} \label{E:matrixSOSprogram_s1}
    \begin{aligned}
        \min_{u}\quad & w^\tr u \\[-1ex]
        \text{subject to} \quad  & P(x) = P_0(x) - \sum_{i=1}^h u_iP_i(x),
        \\ & P(x) \in \Sigma_{n,2d}^r(\mathcal{E},0).
    \end{aligned}
\end{equation}

To exploit the pattern $\mathcal{E}$, we replace $P(x) \in \Sigma_{n,2d}^r(\mathcal{E},0)$ by the stronger condition $P(x) \in  \widetilde{\Sigma}_{n,2d}^r(\mathcal{E},0)$, \emph{i.e.},
\begin{equation} \label{E:matrixSOSprogram_s2}
    \begin{aligned}
        \min_{u}\quad & w^\tr u \\[-1ex]
        \text{subject to} \quad  & P(x) = P_0(x) - \sum_{i=1}^h u_iP_i(x),
        \\ & P(x) \in \widetilde{\Sigma}_{n,2d}^r(\mathcal{E},0).
    \end{aligned}
\end{equation}
Then, Theorem~\ref{T:SOSDecompositionTheorem} allows us to decompose the single large SOS constraint $\widetilde{\Sigma}_{n,2d}^r(\mathcal{E},0)$ with a set of coupled SOS constraints with smaller dimensions.
This can reduce the computational cost of~\eqref{E:matrixSOSprogram_s2} significantly 
if the size of the cliques $\mathcal{C}_k, k = 1, \ldots, t$ is small. %
To demonstrate this, we consider the following special matrix-valued SOS program
\begin{equation} \label{E:matrixSOSprogram_ex}
    \begin{aligned}
        \min_{\gamma}\quad & \gamma \\[-1ex]
        \text{subject to} \quad  & P(x) + \gamma I \; \text{is SOS},
    \end{aligned}
\end{equation}
where $P(x)$ is a $r\times r$ polynomial matrix with an ``arrow'' sparsity pattern, defined as
\begin{equation} \label{E:matrixSOSprogram_exP}
        P(x) = \begin{bmatrix} p_1(x) & p_2(x) & \ldots & p_2(x) \\ p_2(x) & p_3(x) & & \\ \vdots & & \ddots & \\ p_2(x) & & & p_3(x)\end{bmatrix}
\end{equation}
with $p_1(x) = r(x_1^2 + x_2^2 +1), p_2(x) = x_1 + x_2, p_3(x) = x_1^2 + x_2^2 +1$. Note that problem~\eqref{E:matrixSOSprogram_ex} provides an upper bound $-\gamma$ for the minimum eigenvalue of $P(x), \forall x \in \mathbb{R}^n$. Also,
the graph representing the sparsity pattern of $P(x)$ in~\eqref{E:matrixSOSprogram_exP}, illustrated in Fig.~\ref{F:ChordalGraph}(b) for $r = 5$, is chordal and has maximal cliques $\mathcal{C}_k = \{1,k\}, k = 2, \ldots, r$.

We used YALMIP~\cite{lofberg2004yalmip} to solve problem~\eqref{E:matrixSOSprogram_ex} using SDPs, derived both from the usual SOS problem~\eqref{E:matrixSOSprogram_s1} and from the decomposed version of~\eqref{E:matrixSOSprogram_s2}. To solve the SDPs, we used SeDuMi~\cite{sturm1999using} with its default parameters on a PC with a 2.8 GHz Intel Corei7 CPU and 8GB of RAM.
%
%
Table~\ref{T:ResultsTime} lists the CPU time required to solve~\eqref{E:matrixSOSprogram_ex} using either~\eqref{E:matrixSOSprogram_s1} or~\eqref{E:matrixSOSprogram_s2}. We can see that the computational time was reduced significantly when using~\eqref{E:matrixSOSprogram_s2}. This is expected since a single large SOS constraint of dimension $r$ has been replaced by $r-1$ smaller SOS constraints on $2
\times 2$ polynomial matrices. Table~\ref{T:ResultsObjective} shows that using the stronger condition $P(x) \in \widetilde{\Sigma}_{n,2d}^r(\mathcal{E},0)$ brings no conservatism compared to the usual SOS methods for this particular example.


Note that we have $\widetilde{\Sigma}_{n,2d}^r(\mathcal{E},0) \subseteq {\Sigma}_{n,2d}^r(\mathcal{E},0)$ in general. Hence, it may bring certain conservatism when replacing~\eqref{E:matrixSOSprogram_s1} by~\eqref{E:matrixSOSprogram_s2}. This is the case, for instance, when solving~\eqref{E:matrixSOSprogram_ex} for the matrix
$$
        P(x) = \begin{bmatrix} p_1(x) & p_2(x) & p_3(x) \\ p_2(x) & p_4(x) &  0 \\ p_3(x) & 0 & p_5(x)\end{bmatrix},
    $$
    where $$
    \begin{aligned}
        p_1(x) &= 0.8x_1^2 + 0.9x_1x_2+0.3x_2^2 + 1.4x_1 +0.9 x_2+0.8, \\
        p_2(x) &= 0.3x_1+0.91x_2+0.2,\\
        p_3(x) &= 0.1x_1 + x_2+0.8, \\
        p_4(x) &= 0.4x_1^2+1.3x_1x_2+1.1x_2^2+1.4x_1+2.3x_2+1.3, \\
        p_5(x) &= 0.7x_1^2 + 1.3x_1x_2 + 0.9x_2^2 +x_1 + 1.1x_2+0.4.
    \end{aligned}
    $$
Solving the original SOS program~\eqref{E:matrixSOSprogram_s1} returns an objective value $\gamma_1 = 2.007$, while the optimal value of~\eqref{E:matrixSOSprogram_s2} is $\gamma_2 = 2.041$.
The conservatism comes from the fact that we enforce $Q_{23} = 0$ when $p_{23}(x) = 0$. Despite the potential conservatism, the formulation~\eqref{E:matrixSOSprogram_s2} provides a highly scalable way to deal with large sparse matrix-valued SOS programs, as confirmed in Table~\ref{T:ResultsTime}.

\begin{table}
    \centering
    \caption{CPU time (in seconds) required to solve~\eqref{E:matrixSOSprogram_ex} using different formulations.}
    \label{T:ResultsTime}
    \begin{tabular*}{\linewidth}{@{\extracolsep{\fill}}r| rrrrr}
    \toprule[1pt]
     Dimension $r$ & 10 & 20 & 30 & 40  & 50  \\
    \midrule
    Using~\eqref{E:matrixSOSprogram_s1}  & 0.25 & 4.1 & 72.6 & 425.2 & 1\,773.1 \\
    Using~\eqref{E:matrixSOSprogram_s2} &  0.11 & 0.13  &  0.23 & 0.25 & 0.38 \\
    \bottomrule[1pt]
    \end{tabular*}
\end{table}

\begin{table}
    \centering
    \caption{Objective value $\gamma$ for~\eqref{E:matrixSOSprogram_ex} using different formulations.}
    \label{T:ResultsObjective}
    \begin{tabular*}{\linewidth}{@{\extracolsep{\fill}}r| rrrrr}
    \toprule[1pt]
     Dimension $r$ & 10 & 20 & 30 & 40  & 50  \\
    \midrule
    Using~\eqref{E:matrixSOSprogram_s1}  & $-$0.8516 & $-$0.8403 & $-$0.8364 & $-$0.8344 & $-$0.8332 \\
    Using~\eqref{E:matrixSOSprogram_s2} &  $-$0.8516 & $-$0.8403 & $-$0.8364 & $-$0.8344 & $-$0.8332 \\
    \bottomrule[1pt]
    \end{tabular*}
    \vspace{-2mm}
\end{table}

\section{Conclusion} \label{Section:Conclusion}

We have introduced two theorems for the decomposition and completion of sparse SOS matrices. In particular, we proved that a subset of SOS matrices with chordal sparsity patterns can be decomposed into a sum of multiple SOS matrices of smaller dimensions. This property can be easily applied to exploit sparsity in matrix-valued SOS programs.

It should be noted that a notion of \emph{correlative sparsity techniques} has been proposed to exploit chordal sparsity in scalar polynomials~\cite{waki2006sums}. Also, an SOS matrix $P(x)$ can be established via an SOS condition on its scalar representation, \emph{i.e.}, $y^\tr P(x)y$ is SOS in $[x;y]$~\cite{gatermann2004symmetry}. In fact, it is not difficult to show that our decomposition result in Theorem~\ref{T:SOSDecompositionTheorem} corresponds to applying the correlative sparsity technique~\cite{waki2006sums} to the scalar polynomial $y^\tr P(x)y$. However, the scalar interpretation for the SOS completion result (\emph{i.e.}, Theorem~\ref{T:SOSCompletion}) is not clear and requires further investigation.

Finally, our preliminary numerical experiments on simple test problems demonstrate that exploiting chordal sparsity in matrix-valued SOS programs can bring dramatic computational savings at the cost of mild conservatism. Future work will try to confirm these observations in relevant practical applications, \emph{e.g.},~\cite{scherer2006matrix,peet2009positive}.

\balance
\bibliographystyle{IEEEtran}
\bibliography{Reference}

\end{document}